\documentclass[11pt]{amsart}

\usepackage[pdftex]{graphicx} \usepackage{color}

\usepackage[english]{babel} \usepackage[utf8]{inputenc}

\usepackage{microtype}
\usepackage{amsfonts}
\usepackage{verbatim}
\usepackage{amsmath}
\usepackage{amssymb}
\usepackage{amsthm}
\usepackage{extpfeil}
\usepackage{bbm}
\usepackage{paralist}
\usepackage[T1]{fontenc}
\usepackage[colorlinks, citecolor=magenta,urlcolor=red,pdftex]{hyperref}

\usepackage{nicefrac}
\usepackage{microtype}
\usepackage{mathrsfs}

\theoremstyle{plain}

\newtheorem{thm}{Theorem}
\newtheorem{cor}[thm]{Corollary}
\newtheorem{lem}[thm]{Lemma}
\newtheorem*{lem*}{Lemma}

\theoremstyle{definition}

\newcommand{\xqedhere}[2]{%
  \rlap{\hbox to#1{\hfil\llap{\ensuremath{#2}}}}}

\newcommand\Defn[1]{\textbf{\color{blue}#1}}
\newcommand{\cm}[1]{}

\newcommand\mr[1]{\mathrm{#1}}

\newcommand{\R}{\mathbb{R}}

\DeclareMathOperator{\Lk}{Lk}
\DeclareMathOperator{\St}{St}
\DeclareMathOperator{\st}{st}
\DeclareMathOperator{\sd}{sd}

\DeclareMathOperator*{\conv}{conv}

\renewcommand{\phi}{\varphi}

\title[Derived subdivisions induce polytopality]{Derived subdivisions make every 
PL sphere polytopal}
\author{Karim A.~Adiprasito}
\author{Ivan Izmestiev}
\address{Institut des Hautes \'Etudes Scientifiques, Bures-sur-Yvette, France}
\email{adiprasito@ihes.fr, adiprasito@math.fu-berlin.de}
\address{Institut f\"ur Mathematik \\
Freie Universit\"at Berlin \\
Germany}
\email{izmestiev@math.fu-berlin.de}

\keywords{PL sphere; polytopality; derived subdivision}
\subjclass[2010]{57Q05, 52B70, 52B11}

\date{March 20, 2014}
\thanks{K.~Adiprasito has been supported by an EPDI postdoctoral fellowship and 
by the Romanian NASR,
CNCS-UEFISCDI, project PN-II-ID-PCE-2011-3-0533.}
\thanks{I.~Izmestiev has been supported by the European Research Council under 
the European Union's Seventh Framework Programme (FP7/2007-2013)/\allowbreak ERC 
Grant agreement no.~247029-SDModels}

\begin{document}

\begin{abstract}
We give a simple proof that some iterated derived subdivision of every PL sphere 
is combinatorially equivalent to the boundary of a simplicial polytope, thereby 
resolving a problem of Billera (personal communication).
\end{abstract}

\maketitle

\vskip -10mm

\subsection{Making any PL sphere polytopal}
A \Defn{subdivision} of a simplicial complex $\Delta$ is a simplicial complex 
$\Delta'$ with the same underlying space as $\Delta$, such that for every face 
$D'$ of $\Delta'$ there is some face $D$ of $\Delta$ for which $D' \subset D$. 
One also says that $\Delta'$ is a \Defn{refinement} of $\Delta$, or writes 
$\Delta' \prec \Delta$. A \Defn{stellar subdivision} of $\Delta$ at a face 
$\tau$ is defined as
\[
\st(\tau,\Delta):=(\Delta-\tau) \cup \{\conv \{v_\tau\cup \sigma\} : \sigma \in 
\St(\tau,\Delta)-\tau\subset \Delta \}
\]
Here $\Delta-\tau$ denotes the \Defn{deletion} of $\tau$ from $\Delta$, i.e.\ 
the maximal subcomplex of $\Delta$ that does not contain $\tau$, the point 
$v_\tau$ lies anywhere in the relative interior of $\tau$, and 
$\St(\tau,\Delta)$ is the \Defn{star} of $\tau$ in $\Delta$, i.e.\ the minimal 
subcomplex of $\Delta$ that contains all faces of $\Delta$ containing $\tau$. 
Clearly, the combinatorial type of the stellar subdivision does not depend on 
the choice of $v_\tau$.

A \Defn{derived subdivision} $\sd \Delta$ is obtained by stellarly subdividing 
$\Delta$ at all faces in order of decreasing dimension, cf.\ \cite{Hudson}. A 
special case is the \Defn{barycentric subdivision}, where the point $v_\tau$ is 
the barycenter of $\tau$. 

The derived subdivision can be iterated, by defining $\sd^m \Delta:=\sd 
(\sd^{m-1} \Delta)$ and $\sd^0 \Delta:=\Delta$. Our main result in this note is:
\begin{thm}
\label{thm:MakePoly}
For every PL sphere $\Delta$, there exists a $k\ge 0$ such that $\sd^k \Delta$ 
is \Defn{polytopal}, i.e., it is combinatorially equivalent to the boundary 
complex of some convex polytope.
\end{thm}
This answers a question asked to the authors on several occasions, in particular 
by Louis J. Billera (personal communication). The result itself is implicit in 
the work of Morelli \cite[Sec.\ 6]{Morelli}; however, it was never written up 
explicitly. We obtain the following immediate corollary, cf.\
\cite[Cor.\ I.3.12]{AB-SSZ}:
\begin{cor}
For every closed simplicial PL manifold $M$, there is an $n\ge 0$ such that for every 
nonempty face $F$ of $\sd^n M$, the simplicial PL sphere $\Lk(F,\sd^n M)$ is polytopal.
\end{cor}
\begin{proof}
Notice that, if $\Delta$, $\Gamma$ is any pair of PL spheres, then $\sd^n 
(\Delta \ast \Gamma)$ is a stellar subdivision of $\sd^n \Delta \ast  \sd^n 
\Gamma$ (where $\ast$ denotes the join operation); since stellar subdivisions 
preserve polytopality, we therefore observe that $\sd^n (\Delta \ast \Gamma)$ is 
polytopal if $\sd^n \Delta$ and $\sd^n \Gamma$ are polytopal. 

Observe secondly that if $\Delta$ is any simplicial complex, and $F$ is any face 
of $\sd \Delta$, then there is a face $\widetilde{F}\in \Delta$ and simplices $\sigma_1$, $\cdots$, $\sigma_k$ 
such that \[\Lk(F,\sd \Delta)\, \cong\, \sd \partial \sigma_1\ast \cdots\ast \sd \partial\sigma_k \ast \sd \Lk(\widetilde{F},\Delta).\]
Now, let $n$ be chosen large enough such that for all faces $\widetilde{F}$ of 
$M$, the complex $\sd^n \Lk(\widetilde{F},M)$ is polytopal. It then follows from 
the two observations above that for every face $F$ of $\sd^n M$, the complex 
$\Lk(F,\sd^n M)$ is polytopal.
\end{proof}

\begin{proof}[\textbf{Proof of Theorem \ref{thm:MakePoly}}]
A \Defn{complete pointed fan} in $\R^{d+1}$ is a partition of $\R^{d+1}$ into 
convex polyhedral cones with apices at the origin $\mathbf{0}$ such that the 
intersection of any two cones is a face of both.
A fan is called \Defn{regular} if it consists of the cones over the faces of a 
convex polytope (with the origin in the interior). A fan $F$ is regular if and 
only if there exists a PL function $\phi \colon \R^{d+1} \to \R$ whose domains 
of linearity are exactly the full-dimensional cones of $F$ and that is strictly 
convex across every codim 1 cone of $F$, cf.\ \cite{LRS}. Thus we have to show 
that every PL $d$-sphere $\Delta$ becomes combinatorially equivalent to some 
regular simplicial fan after several derived subdivisions.

We will be repeatedly using the following simple observation:
\begin{lem}[cf.\ {\cite[Ch.\ 1, Lem.\ 4]{ZeemanBK}}]
Let $\Delta_1$ and $\Delta_2$ denote two simplicial complexes with the same 
support; then there is a derived subdivision $\sd^k\Delta_1$ that refines 
$\Delta_2$. Moreover, one can choose $k\le |f|(\Delta_2)$.
\end{lem}
Here, $|f|(\cdot)$ denotes the total number of faces of a simplicial complex.

\noindent\textbf{Claim 1:} There is an $n$ such that $\sd^n \Delta$ is 
combinatorially equivalent to a simplicial (not necessarily regular) fan in 
$\R^{d+1}$.

By definition, $\Delta$ is PL homeomorphic to the boundary of the 
$(d+1)$-simplex $\sigma^{d+1}$. In other words, there are combinatorially 
equivalent subdivisions
$\widetilde{\Delta}$ and $\Sigma$ of $\Delta$ and $\partial \sigma^{d+1}$, 
respectively. Let now 
\[\vartheta:\widetilde{\Delta} \longrightarrow \Sigma\] 
denote a facewise linear map realizing the combinatorial equivalence, and let 
$\sd^n \Delta$ be chosen fine enough such that $\sd^n \Delta 
\prec\widetilde{\Delta}$. Then $\vartheta (\sd^n \Delta)$ is a subdivision of 
$\partial \sigma^{d+1}$ combinatorially equivalent to $\sd^n \Delta$. The cone 
with respect to any interior point of $\sigma^{d+1}$ gives the desired 
simplicial fan.

In the following we identify subdivisions of $\partial \sigma^{d+1}$ with the 
corresponding fans.

\noindent\textbf{Claim 2:} There is a regular subdivision $\Delta'$ of $\sd^n 
\Delta$.

Regularity is preserved under stellar, and in particular derived, subdivisions, 
cf.\ \cite{LRS}. Thus we may take for $\Delta'$ any derived subdivision of 
$\partial \sigma^{d+1}$ that refines $\sd^n \Delta$.

Let $\sd^{n+m} \Delta$ be a derived subdivision that refines the regular 
subdivision $\Delta'$:
\[\sd^{n+m} \Delta\, \prec\, \Delta'\, \prec\, \sd^n \Delta
\]
\noindent \textbf{Claim 3:} $\sd^{n+m} \Delta$ is regular.

We have to show that there is a convex PL function with the fan $\sd^{n+m} 
\Delta$. First, let us construct a PL function $h:\R^{d+1} \longrightarrow \R$ 
that is linear on the faces of $\sd^{n+m} \Delta$, and that is strictly convex 
at every $\mr{codim}\  1$-face \emph{except} at the $\mr{codim}\  1$ skeleton of 
$\sd^n \Delta$.

This is proven by induction: $\sd^{n+m} \Delta$ is a derived, and in particular 
an iterated stellar subdivision of $\sd^n \Delta$; let $\Delta_1$, $\Delta_2, 
\ldots$ denote the intermediate complexes, so that $\Delta_{i+1}$ is obtained 
from $\Delta_{i}$ using a single stellar subdivision (obtained by introducing a 
vertex $\nu_i$). 

If $\nu$ is a ray of a simplicial fan $\mr{F}$, then let us denote by 
\[[\nu,\mr{F}]^\ast(\cdot):\R^{d+1} \longrightarrow \R\] the function that is 
$\langle \cdot, \nu\rangle$ on the ray spanned by $\nu$, that is $0$ on all 
other rays and that is linear on the faces of $\mr{F}$. Note that 
$[\nu,\mr{F}]^\ast(\cdot)$ is strictly convex across all $\mr{codim}\ 1$-faces 
of $F$ that contain $\nu$.

On $\sd^n \Delta=\Delta_0$, we just take the zero function.

By induction assumption, let us assume that $\Delta_i$ admits a function 
$h_i:\R^{d+1} \longrightarrow \R$ that is linear on faces of $\Delta_i$, and 
strictly convex at every $\mr{codim}\  1$ face except those in the $\mr{codim}\  
1$ skeleton of $\Delta$. Then, for $\varepsilon_i>0$ small enough, the function
$\varepsilon_i[\nu_i,\Delta_{i+1}]^\ast+h_i$ is linear on every face of 
$\Delta_{i+1}$, and strictly convex at all $\mr{codim}\  1$ faces newly 
introduced. Hence, by induction, there is a function $h$ with the desired 
property.
\enlargethispage{4mm}

Now, $\Delta'$ is regular, and hence there exists a strictly convex piecewise 
linear function $h': \R^{d+1} \longrightarrow \R$ whose domains of linearity are 
the facets of $\Delta'$. In particular, $h'$ is linear on all faces of 
$\sd^{n+m} \Delta$. The function $h'$ is strictly convex across those faces 
where the convexity of $h$ can fail. Hence, for an $\varepsilon>0$ small enough, 
$\varepsilon h+h'$ is strictly convex at all codimension one faces of $\sd^{n+m} 
\Delta$, and linear on all facets of $\sd^{n+m} \Delta$. This finishes the 
proof. 
\end{proof}

\subsection{Algorithmic aspects}
Now that we determined that sufficiently many iterations of the derived 
subdivision make any PL sphere polytopal, it makes sense to ask how many 
precisely are needed.
If $\dim \Delta=2$, then $\Delta$ is combinatorially equivalent to the boundary 
of a convex polytope by Steinitz Theorem, cf.\ \cite{Z}; the fact that the graph 
of every triangulation of $S^2$ is $3$-connected is an easy exercise. Thus $k=0$ 
suffices in this case.

For higher dimensions, $k$ can not be bounded that easily, as we shall see now. 
As usual, $f_i(\cdot)$ denotes the number of $i$-dimensional faces of a 
simplicial complex.
\begin{thm}
\begin{compactenum}[\rm (a)]
\item If $d \ge 3$, then there is no $k$ that would depend only on $d$ such that 
all PL $d$-spheres become polytopal after $k$ derived subdivisions.
\item For $d=3$, it number $k=k(\Delta)$ of derived subdivisions needed to make 
a PL sphere $\Delta$ polytopal can be bounded from above by \[k(\Delta)\, \le\, 
a\cdot 2^{b\cdot f_3(\Delta)\cdot 2^{c\cdot f_3^2(\Delta) \cdot 2^{d\cdot f_3^2(\Delta)}}} 
\, +\, c\cdot f_3^2(\Delta) \cdot 2^{d\cdot f_3^2(\Delta)},\]
where $a, b, c, d\ge 0$ are constants independent of $\Delta$.
\item If $d \ge 5$, then the number of derived subdivisions that makes a PL 
$d$-sphere $\Delta$ polytopal is not (Turing machine) computable from $\Delta$. 
\end{compactenum}
\end{thm}

In other words, if $\varphi:\mathfrak{K}_d\longmapsto \mathbb{N}$ is any 
computable function (cf.\ \cite{Davis}), where $\mathfrak{K}_d$ is the 
collection of $d$-dimensional simplicial complexes, then for some PL $d$-sphere 
$\Delta$, more than $\varPhi(\Delta)$ derived subdivisions are needed to make 
$\Delta$ polytopal. In particular, the number of subdivisions is not computable 
from the dimension, the $f$-vector, the flag vector or any other combinatorial 
invariant of $\Delta$.

\begin{proof}
\begin{asparaenum}[\rm (a)]
\item The first statement follows from the work of Bing \cite{Bing} and 
Lickorish \cite{LME}. Indeed, one can show that for every $d\ge 3$ and every 
$k\ge 0$, there is a PL $d$-sphere $\Delta$ such that $\sd^k \Delta$ is not 
shellable (cf.\ \cite{LME}). Since the boundary of every convex polytope is 
shellable \cite{BruggesserMani}, $\sd^k \Delta$ can not be combinatorially 
equivalent to the boundary of a convex polytope. Compare also \cite{AB-SSZ}.

\item For the second assertion, recall that there is an $\ell$ such that 
$\sd^\ell \Delta$ is combinatorially equivalent to a subdivision of (the 
simplicial fan spanned by) $\partial \sigma^4$ by Claim 1 in the proof of 
Theorem \ref{thm:MakePoly}. By a result of Mijatovi\'c \cite{MJT}, $\ell$ 
can be bounded in terms of the number of faces of $\Delta$; more explicitly, one 
can show that $\ell\le c' \cdot f_3^2(\Delta)\cdot 2^{d'\cdot f_3^2(\Delta)}$, 
where $c', d'\ge 0$ are constants independent of $\Delta$.

Now, there is a iterated derived subdivision $\Delta'=\sd^m \partial \sigma^4$ 
of $\partial \sigma^4$ that is regular and subdivides $\sd^\ell \Delta$, and the 
number of derived subdivisions needed can be bounded from above by 
\[m\le|f|(\sd^\ell \Delta)\le (4!)^\ell \cdot 2^4 \cdot f_3(\Delta).\] Finally, 
the fan $\Delta'$ is regular, and there is an $n$ such that $\sd^{\ell+n} 
\Delta$ subdivides $\Delta'$, and \[n\le |f|(\Delta')\le (4!)^m \cdot 2^4 \cdot 
f_3(\partial \sigma^4).\]
But $\sd^{\ell+n} \Delta$ is regular by Claim 3 in the proof of Theorem 
\ref{thm:MakePoly}. 

\item For the final claim: suppose there exists a Turing machine computable 
function $\varphi:\mathfrak{K}_d\longmapsto \mathbb{N}$ that, for every PL 
$d$-sphere $\Delta$, $d\ge 5$, returns a value $\varphi(\Delta)$ such that 
$\sd^{\varphi(\Delta)} \Delta$ is polytopal, we would also have Turing machine 
that decides whether or not a given simplicial $d$-manifold, $d\ge 5$, is a PL 
sphere: Recall that deciding whether a given simplicial complex is the boundary 
of a convex polytope is complete within the existential theory of the reals, and 
therefore Turing machine decidable cf.\ \cite{Davis, Mnev}. Now, if this Turing 
machine returns, for any $d$-dimensional simplicial complex $\Delta$, that 
$\sd^{\varphi(\Delta)} \Delta$ is not polytopal, then $\Delta$
is not a PL sphere by assumption; if instead it returns that 
$\sd^{\varphi(\Delta)} \Delta$ is polytopal, then $\Delta$ is a PL sphere, as 
desired.

The existence of this Turing machine, however, stands in contradiction to a 
classical result of S.~P.~Novikov \cite{Novikov}, cf.\ \cite{Nabutovsky}, who 
proved that it is not decidable whether a given $5$-manifold is actually the PL 
$5$-sphere. Therefore, the assumption is wrong, and no such Turing machine 
bounding $k$ exists. \qedhere
\end{asparaenum}
\end{proof}

\subsection{Regular triangulations and geometric bistellar moves}
Let $P \subset \R^d$ be a convex $d$-polytope. A triangulation $T$ of $P$ (the 
vertex set of $T$ may be bigger than that of $P$) is called \emph{regular} if 
there exists a PL function $h \colon P \to \R$ linear on all $d$-simplices of 
$T$ and convex across all of its $(d-1)$-simplices, compare also the notion of a 
regular fan in the proof of Theorem \ref{thm:MakePoly}, and \cite{Z} or 
\cite{LRS}.

While polytopality is a combinatorial property of a simplicial complex, 
regularity of a triangulation depends not only on its combinatorics, but also on 
the position of its vertices.
\begin{thm}
\label{thm:MakeReg}
For every triangulation $T$ of a convex polytope $P$ there is a $k$ such that 
some derived subdivision $\sd^k T$ is regular.
\end{thm}
\begin{proof}
The proof is similar to that of Theorem \ref{thm:MakePoly}.
We need a regular triangulation of $P$ to start with: To find one, choose $h_i 
\in \R$ for every vertex $p_i$ of $P$ generically and take the lower envelope of 
the points $(p_i, h_i) \in \R^{d+1}$, cf.\ \cite{LRS}. By applying derived 
subdividisions to $T'$, we see that there exists a regular triangulation $T'$ of 
$P$ that refines $T$. Now, there is an $m\ge 0$ such that $\sd^m T$ refines 
$T'$. It now follows as in the proof of Theorem \ref{thm:MakePoly}, Claim $3$, 
that $\sd^m T$ is regular.
\end{proof}

\begin{cor}
\label{cor:Pachner}
Any two triangulations $T_0$ and $T_1$ of $P$ can be connected by a sequence of 
geometric Pachner (or bistellar) moves.
\end{cor}

This is essentially the main result of \cite{Morelli} and \cite{Wlo}, with the 
difference that we don't assume $P$ to be a lattice polytope and don't require 
triangulations to be unimodular. Ewald and Shephard had earlier proven it for 
regular triangulations \cite{ES}. On the other hand, Pachner \cite{Pachner}
proved that PL homeomorphic manifolds are related by \emph{combinatorial} 
Pachner moves.

To deduce Corollary \ref{cor:Pachner} from Theorem \ref{thm:MakeReg}, take any 
triangulation $\widetilde{T}$ of $P \times [0,1]$ that restricts to $T_0$ and 
$T_1$ on $P \times \{0\}$ and $P \times \{1\}$ respectively, and apply derived 
subdivisions to make $\widetilde{T}$ regular. Sweeping out from $0$ to $1$ 
produces a sequence of bistellar moves. Details can be found in \cite[Sec.\ 
2]{IzmScl}. 

Note that geometric bistellar \emph{flips} (bistellar moves other than stellar 
subdivisions and moves inverse to them) do not suffice in general to transform 
one of two triangulations of the same point configuration into the other; see 
\cite{Santos05, Santos06} for a counterexample in dimension $5$.

\small{\bibliographystyle{myamsalpha}
\bibliography{Ref}}

\end{document}